\documentclass[12pt]{amsart}
\usepackage{amsthm,amsmath,amssymb,mathrsfs,amsbsy}
\usepackage{comment}
\usepackage[utf8]{inputenc}
\usepackage[english]{babel}
\usepackage[colorlinks,linkcolor = red]{hyperref}
\usepackage[
backend=biber,
style=alphabetic,
sorting=ynt
]{biblatex}
 
\hypersetup{ linktoc=page }
\newtheorem{theorem}{Theorem}[subsection]
\newtheorem*{theoremA}{Main Theorem}
\newtheorem{definition}[theorem]{Definition}
\newtheorem{claim}[theorem]{Claim}

\newtheorem{proposition}[theorem]{Proposition}

\newtheorem{notation}[theorem]{Notation}

\newtheorem{remark}[theorem]{Remark}

\newcommand{\lcs}[1]{\mathfrak{n}_{#1}}
\newcommand{\rs}[1]{\mathfrak{g}_{#1}}
\catcode`\@=11
\newdimen\cdsep
\def\cdstrut{\vrule height .6\cdsep width 0pt depth .4\cdsep}
\def\@cdstrut{{\advance\cdsep by 2em\cdstrut}}

\def\arrow#1#2{
	\ifx d#1
	\llap{$\scriptstyle#2$}\left\downarrow\cdstrut\right.\@cdstrut\fi
	\ifx u#1
	\llap{$\scriptstyle#2$}\left\uparrow\cdstrut\right.\@cdstrut\fi
	\ifx r#1
	\mathop{\hbox to \cdsep{\rightarrowfill}}\limits^{#2}\fi
	\ifx l#1
	\mathop{\hbox to \cdsep{\leftarrowfill}}\limits^{#2}\fi
}
\catcode`\@=12

\title[Characterising Semi-Simple Lie Algebras]{Characterising Semi-Simple Lie Algebras by Their Borel Nilpotent Radical}
\author{Guy Kapon, Lior Hadassy}

\begin{document}

\maketitle

\begin{abstract}
We show that semi-simple lie algebras can be characterized by their maximal nilpotent subalgebra, which is the same as the nilpotent radical of a Borel subalgebra.
\end{abstract}
\section{Introduction}
The classification of simple lie algebras is well known and is made in terms of the combinatorics of root systems. 

While this is satisfactory for many purposes, one might wonder if they can also be classified by more algebraic means, or at least characterized that way. Indeed, such a theorem can perhaps be extended to cases where the theory of root system is less applicable, like the classification of finite groups.

The first option is to characterize the algebra from a Borel subalgebra, which is quite easy. A less obvious question is the question if one can do this from the nilpotent radical of a Borel subalgebra, which is also a maximal nilpotent subalgebra, see \cite{MNA}, and we prove in this note:
\begin{theoremA}
\label{thmA}
Let $\rs{1},\rs{2}$ be two semi-simple lie algebras, let $\mathfrak{b}(\rs{1}),\mathfrak{b}(\rs{2})$ be Borel subalgebras, and let $\lcs{}(\rs{i})$ be the nilpotent radical of $\mathfrak{b}(\rs{i})$. Assume that $\lcs{}(\rs{1})\cong \lcs{}(\rs{2})$ then $\rs{1} \cong \rs{2}$.
\end{theoremA}
Originally, we proved this theorem for the simple case, while the semi-simple case does contain the simple case, it seemed useful to keep the simple case as a more hands-on example that demonstrates the ideas and calculations that appear in the general proof.

\subsection{Idea of The Proof}
The nilpotent radical is a nilpotent algebra, distinguishing between such algebras is relatively simple, as by definition they have a lot of ideals and therefore invariants.

For the simple case, our main tool is the lower central sequence.

It is defined by $\lcs{1}=\lcs{}, \lcs{i+1}=[\lcs{i},\lcs{}]$, and is obviously an invariant of the algebra. 

Already $\rm{dim}(\lcs{1}/\lcs{2})$ recovers the rank of the original simple algebra, and together with the dimension of the algebra this almost completes the proof. 

To deal with the remaining cases we consider the product $\lcs{i}/\lcs{i+1} \times \lcs{j}/\lcs{j+1}\rightarrow \lcs{i+j}/\lcs{i+j+1}$ induced by the lie bracket, and show that in the remaining cases these are not always the same.

For the semi-simple case we have to add the upper central sequence $\mathfrak{u}_{i}$ which is defined in a similar way to $\lcs{i}$, and also the analogous product $\lcs{i}/\lcs{i+1} \times \mathfrak{u}_{j}/\mathfrak{u}_{j-1} \rightarrow \mathfrak{u}_{j-i}/\mathfrak{u}_{j-i-1}$.

\section{The Simple Case}
We follow the notations of \cite{AK}, everything is done over $\mathbb{C}$.

We have a (semi-)simple algebra $\rs{}$, a Cartan subalgebra $\mathfrak{h}(\rs{})$, its root system is denoted $R(\rs{})$, and its root spaces are $\rs{\alpha}$, we will assume that some Borel subalgebra $\mathfrak{b}(\rs{})$ containing $\mathfrak{h}(\rs{})$ was chosen, and we denote by $\mathfrak{n}(\rs{})=[\mathfrak{b}(\rs{}),\mathfrak{b}(\rs{})]$ its nilpotent radical.
We have: $$\mathfrak{n}(\rs{}) =\bigoplus_{\alpha \in R_{+}} \rs{\alpha} $$

Where $R_{+}$ is the set of positive roots with respect to our chosen Borel subalgebra. When there is no confusion about the simple lie algebra, we will drop it from the notation.

Our goal is to prove that we can characterize a simple lie algebra from the structure of $\mathfrak{n}$. Two invariants that we will explain now will suffice.
\subsection{Lower Central Series}
The Lower Central Series of $\mathfrak{n}$ is defined by $\lcs{1} = \lcs{} $ and $\lcs{i+1} = [\lcs{i},\mathfrak{n}]$. 

Let us calculate the lower central series in our case.

For a positive root $\alpha \in R_+$ we define its degree $\rm{deg}(\alpha)$ as the sum of coefficients when we write $\alpha$ as a sum of simple roots.

\begin{claim}
For $i\geq 1$, $\lcs{i}$ is the subspace generated by $\rs{\alpha}$ for $\alpha$ with $\rm{deg}(\alpha)\geq i$.
\end{claim}
\begin{proof}
For $i=1$ this is obvious, we proceed by induction.

Let us denote the subspace generated by $\rs{\alpha}$ for $\rm{deg}(\alpha)\geq i$ by $\lcs{i}'$, and assume we have proved $\lcs{i} = \lcs{i}'$. 

Since $[\rs{\alpha},\rs{\beta}]\subseteq \rs{\alpha + \beta}$, and the representation of a root as a sum of simple roots is unique, it is obvious that $\lcs{i+1}=[\lcs{i},\lcs{}]\subseteq \lcs{i+1}'$. 

On the other hand, let $\alpha$ be a root with $\rm{deg}(\alpha)\geq i+1$, then we know that there is a simple root $\beta$ s.t $\alpha - \beta$ is also a root, in that case $[\rs{\alpha-\beta},\rs{\beta}]=\rs{\alpha}$, obviously $\rm{deg}(\alpha-\beta)= \rm{deg}(\alpha)-1\geq i$ Since $\rm{deg}$ is linear, and so $[\rs{\alpha-\beta},\rs{\beta}]\subseteq [\lcs{i},\lcs{}]=\lcs{i+1}$ and we are done.
\end{proof}
The lower central series has an additional structure in our case:
\begin{notation}
For $i\geq 1$ denote $\lcs{}^{i} = \lcs{i}/\lcs{i+1}$.
\end{notation}
\begin{claim}
For any $i,j\geq 1$ we have a bilinear map $\lcs{}^{i} \times \lcs{}^{j} \rightarrow \lcs{}^{i+j}$
\end{claim}
\begin{proof}
By our description of $\lcs{i}$ it's obvious that $[\lcs{i},\lcs{j}]\subseteq \lcs{i+j}$, which therefore defines a product $\lcs{i}\times\lcs{j}\rightarrow \lcs{i+j}$.

We can then quotient out $\lcs{i+j+1}$, and it only remains to see that $\lcs{i+1}\times \lcs{j}$ and $\lcs{i}\times \lcs{j+1}$ go to zero, but we already explained that their commutator is in $\lcs{i+j+1}$ as required.
\end{proof}
\begin{remark}
While it may seem we used our description of $\lcs{i}$, we only used this for $[\lcs{i},\lcs{j}]\subseteq \lcs{i+j}$ which is true in general by the Jacobi identity and induction. 
\end{remark}
\begin{remark}
Both of the claims can be described by saying that $gr(\lcs{} ):=\bigoplus_{i} \lcs{}^{i}$ is a lie algebra by itself, which in our case is in fact (non-canonically) isomorphic to $\lcs{}$.
\end{remark}

\subsection{Rank}
From the previous section we see that $\lcs{}^{i}$ is isomorphic to the space generated by $\rs{\alpha}$ with $\rm{deg}(\alpha)= i$, so the number of such roots is an invariant of $\lcs{}$.

In particular, the number of roots of degree $1$, i.e. simple roots, is an invariant which is equal to the rank of the original lie algebra.

This means we only have to distinguish between $A_{n},B_{n},C_{n},D_{n}$ and the exceptionals for a constant $n$.

Now we can compare dimensions, since $\rm{dim}(\rs{})=2\cdot \rm{dim}(\lcs{})+n$ we will compare the dimension of $\rs{}$ instead of the dimension of $\lcs{}$, these numbers are well known:
\begin{center}
\begin{tabular}{||c c c||} 
 \hline
 Lie Algebra & Rank & Dimension \\ [0.5ex] 
 \hline\hline
 $A_n$ & $n$ & $n(n+2)$ \\ 
 \hline
 $B_n$ & $n$ & $n(2n+1)$ \\ 
 \hline
 $C_n$ & $n$ & $n(2n+1)$ \\ 
 \hline
 $D_n$ & $n$ & $n(2n-1)$ \\ 
 \hline
 $E_6$ & $6$ & $78$ \\ 
 \hline
 $E_7$ & $7$ & $133$ \\ 
 \hline
 $E_8$ & $8$ & $248$ \\ 
 \hline
 $F_4$ & $4$ & $52$ \\ 
 \hline
 $G_2$ & $2$ & $14$ \\  [1ex] 
 \hline
\end{tabular}
\end{center}
This table can be found for example at \cite[Table 2]{LLL}.

From the table, we easily see that $B_n,C_n$ always have the same dimension, which is always different from that of $D_n$. 

$A_n$ has the same dimension as $B_n$ or $C_n$ only for $n=1$, and the same dimension as $D_n$ only for $n=3$, both of which corresponds to the well-known facts that $A_{1}=B_{1}=C_{1}$ and $A_{3}=D_{3}$.

Finally, the only exceptional that has the same dimension as a classic algebra of the same rank is $E_{6}$ with $B_{6}$ and $C_{6}$.

In summary, we proved:
\begin{proposition}
\label{Prop1}
If $\rs{1},\rs{2}$ are two simple lie algebras, and we have $\lcs{}(\rs{1}) \cong \lcs{}(\rs{2})$ then either $\rs{1} \cong \rs{2}$, $\{\rs{1},\rs{2}\}\cong\{B_{n},C_{n}\}$ or one of the algebras is $E_{6}$ and the other one is $B_{6}$ or $C_{6}$.
\end{proposition}

The next natural step is to look at the number of roots of height $i$ for a general $i$. 

It turns out that those numbers are the same for $B_{n}$ and $C_{n}$, but do distinguish between $E_{6}$ and $B_{6},C_{6}$.

Specifically $E_{6}$ has $5$ roots of degree $4$ while $B_{6},C_{6}$ have $4$ such roots. 

Both claims can be checked by writing down root systems. For $B_{6},C_{6}$ we will do this later.One can also look at \cite[Appendix]{CMR} where root posets for those root systems are drawn, in particular, one can see this result.

\subsection{$B_{n}$ and $C_{n}$}
As we mentioned, the number of roots of each degree for $B_{n}$ and $C_{n}$ is the same. Still, the combinatorics of the roots is different, and we claim that our bilinear product will distinguish between them.

First, we write down the root systems and our choice of simple roots. Both are written inside the vector space with basis $\{e_{i}\}_{1\leq i\leq n}$ like in \cite[205-206]{AK}:

$$R(B_{n}) = \{ \pm e_{i} \pm e_{j}\}_{i\neq j}\cup \{ \pm e_{i} \}$$
$$\Pi(B_{n}) =\{ e_{i} - e_{i+1}\}_{1\leq i \leq n-1}\cup \{e_{n}\}$$

$$R(C_{n}) = \{ \pm e_{i} \pm e_{j}\}_{i\neq j}\cup \{ \pm 2e_{i} \}$$
$$\Pi(C_{n}) = \{ e_{i} - e_{i+1}\}_{1\leq i \leq n-1}\cup \{2e_{n}\}$$

Now let us consider in both cases the bilinear product:
$$ \lcs{}^{2} \times \lcs{}^{2n-3} \rightarrow \lcs{}^{2n-1} $$
Notice that in both cases $\lcs{2n-1}$ is $1$ dimensional, while $\lcs{2n}=0$. In $B_{n}$ the root of highest degree is $e_{1}+e_{2}$ with degree $2n-1$, and in $C_{n}$ it is $2e_{1}$ also with degree $2n-1$.

Next we see that $\lcs{}^{2n-3}$ is two dimensional in both cases, in the case of $B_{n}$ the roots of degree $2n-3$ are $e_{1}+e_{4},e_{2}+e_{3}$ and for $C_{n}$ there are $e_{1}+e_{3},2e_{2}$.

Here we have assumed $n\geq 4$, since we had $e_{1}+e_{4}$, this discussion however is also true for $n=3$, by setting $e_{4}=0$ in this case. 

For $n=2$ or $n=1$ the dimension calculation is wrong, but in those cases $B_{2}=C_{2},B_{1}=C_{1}$ so there is nothing to prove. 

\begin{claim}
In the case of $C_{n}$, $n\geq 3$, this bilinear product has non-trivial right kernel, specifically it contains $\rs{2e_{2}}$.
\end{claim}
\begin{proof}
For any $\alpha$ of degree $2$, $\alpha+2e_{3}\notin R$,  if this was not the case, $\alpha + 2e_{3}$ would have to be $2e_{1}$, because its degree is $2n-1$, but $2e_{1}-2e_{2}$ is not a root.

This means that $[\rs{\alpha},\rs{2e_{2}}] = \rs{\alpha + 2e_{2}}=0$ for any root of degree $2$, but those $\rs{\alpha}$ exactly generate $\lcs{}^{2}$.
\end{proof}

So now to finish we only have to prove:
\begin{claim}
In the case of $B_{n}$, $n\geq 3$, the bilinear product is non-degenerate on the right.
\end{claim}
\begin{proof}
Similar to before, notice that $(e_{1}+e_{4}) + (e_{2}-e_{4})=(e_{2}+e_{3})+(e_{1}-e_{3})=e_{1}+e_{2}$, which means that
\begin{center}
\begin{tabular}{ c c  }
 $[\rs{e_{2}-e_{4}}\rs{e_{1}+e_{4}}]=\rs{e_{1}+e_{2}}$ & $[\rs{e_{1}-e_{3}}\rs{e_{1}+e_{4}}]=0$  \\  
 \\
 $[\rs{e_{2}-e_{4}}\rs{e_{2}+e_{3}}]=0$ &$[\rs{e_{1}-e_{3}}\rs{e_{2}+e_{3}}]=\rs{e_{1}+e_{2}}$    
\end{tabular}
\end{center}
Which obviously means that the bilinear product is non-degenerate on the right.

(For $n=3$ the same calculation holds after erasing $e_{4}$)
\end{proof}

So now we are done.
\begin{proof}[Proof of The Simple Case] By proposition \ref{Prop1} with addition to the remarks after it, it only remains to show that $\lcs{}(B_{n})$ is not isomorphic to $\lcs{}(C_{n})$ for $n\geq 3$ as for $n\leq 2$, $B_{n}\cong C_{n}$.

If they were isomorphic then we also would have had an isomorphism of the lower central series, and the bilinear product we constructed on it. but we just showed that for $B_{n}$ one of those products is non-degenerate on the right, while for $C_{n}$ the same one is, and therefore they are not isomorphic.
\end{proof}
One notice that the calculation of the right kernel was completely combinatoric, this is true in general, and in the semi-simple case, we will not go into details about such calculations.

\medskip
\section{The Semi-Simple Case}
For the semi-simple case, we have to distinguish between direct sums of simple algebras, all the invariants that we had before are still invariants, their value is the sum over the simple algebras of the corresponding invariant. 

For example $\rm{dim}(\lcs{}^1)$ will be the sum over the simple algebras, of their rank. 

To use this we will need another set of invariants which in the previous case was redundant.
\subsection{The Upper Central Series} The upper central series is the dual notion to the lower central series:

\begin{definition}
$\mathfrak{u}_{0} = 0$, and $\mathfrak{u}_{i+1}$ is defined as the unique subalgebra containing $\mathfrak{u}_{i}$ such that $\mathfrak{u}_{i+1}/\mathfrak{u}_{i} = Z(\lcs{}/\mathfrak{u}_{i}) $ where $Z$ denotes the center of a lie algebra.
\end{definition}

Like the lower central series, this is an invariant, we had no need of it in the case of a simple algebra, because of:
\begin{claim}
If $\lcs{}$ is the nilpotent radical of a Borel subalgebra of a simple lie algebra, then the lower central series and the upper central series are the same up to a change of indexes.
\end{claim}
As we already have a description for the lower central series, we only have to show that this description holds for $\mathfrak{u}_{i}$.
Let $d$ be the degree of the highest root of $\rs{}$ then a more precise statement can be written:
\begin{claim}
For $i\geq 0$, $\mathfrak{u}_{i}$ is the subspace generated by $\rs{\alpha}$ for $\alpha$ with $\rm{deg}(\alpha) > d-i$.
\end{claim}

\begin{proof}
Again, we show this by induction, for $i=0$ this is obvious.
Assuming the result is true for $i$, we have to show it for $i+1$.

If $\rm{deg}(\alpha)>d-i-1$ and $\beta$ is any root, then $[\rs{\alpha},\rs{\beta}]\subseteq \rs{\alpha+\beta}\subseteq \mathfrak{u}_{i}$. This is true because $\rm{deg}(\alpha+\beta)=\rm{deg}(\alpha) + \rm{deg}(\beta)>(d-i-1)+1=d-i$ together with the induction hypothesis. This shows that $\rs{\alpha} \subseteq \mathfrak{u}_{i+1}$.

On the other hand, assume $v\in \mathfrak{u}_{i+1}$, and that its projection to some $\rs{\alpha}$ with $\rm{deg}(\alpha)\leq d-i-1$ is not zero. Since $\alpha$ is not the highest root, we know that we can find some simple root $\beta$ such that $\alpha+\beta$ is also a root. Then we see that $[\rs{\beta},v]$ has non-trivial projection to $\rs{\alpha+\beta}$, but by the induction hypothesis $\rs{\alpha+\beta}$ is not inside $\mathfrak{u}_{i}$, which means $v$ is not in the center of $\lcs{}/\mathfrak{u}_{i}$, and this is a contradiction. 
\end{proof}
\begin{notation}
For $i\geq 1$ denote $\mathfrak{u}^{i} = \mathfrak{u}_{i}/\mathfrak{u}_{i-1}$.
\end{notation}

While this claim might seem to suggest that the upper central series holds no more information than the lower central series, as they are just a reordering of each other, this is not the case.

The critical point is that the reordering depends on the simple algebra, so when we sum those for different algebras, they will hold different information.

For example, while $\lcs{}^{1}$ dimension is equal to the number of simple roots in any simple algebra, the last non-zero $\mathfrak{u}^{i}$ has dimension equal to the number of simple roots in the algebras with the maximal degree of the highest root.

Before concluding this section let us quickly mention that we have an analogue  of the bilinear product:
\begin{claim}
The lie bracket induces a product $\lcs{}^{i} \times \mathfrak{u}^{j} \rightarrow \mathfrak{u}^{j-i}$. (When $j-i \leq 0$, $\mathfrak{u}^{j-i}$ is interpreted as zero.)
\end{claim}

\begin{proof}
Like the previous product we just need to show that the bracket takes $\lcs{i}\times \mathfrak{u}_{j}\rightarrow \mathfrak{u}_{j-i}$.

For $i=1$ this is obvious, as $[\lcs{},\mathfrak{u}_{j}]\subseteq \mathfrak{u}_{j-1}$ by definition.

Using induction we get: $$[\lcs{i+1},\mathfrak{u}_{j}] =[[\lcs{i},\lcs{}],\mathfrak{u}_{j}]\subseteq [\lcs{i}, [\lcs{},\mathfrak{u}_{j}]] + [\lcs{}, [\lcs{i},\mathfrak{u}_{j}]] \subseteq $$
$$ \subseteq [\lcs{i},\mathfrak{u}_{j-1}] + [\lcs{},\mathfrak{u}_{j-i}]\subseteq \mathfrak{u}_{j-i-1}$$
\end{proof}

\subsection{Proof}
Now we wish to use the invariants we constructed to prove:
\begin{theoremA}
Let $\rs{1},\rs{2}$ be two semi-simple lie algebras, let $\mathfrak{b}(\rs{1}),\mathfrak{b}(\rs{2})$ be Borel subalgebras, and let $\lcs{}(\rs{i})$ be the nilpotent radical of $\mathfrak{b}(\rs{i})$. Assume that $\lcs{}(\rs{1})\cong \lcs{}(\rs{2})$ then $\rs{1} \cong \rs{2}$.
\end{theoremA}
First let us introduce a notation, any semi-simple lie algebra $\rs{}$ is a sum of simple lie algebra, we will denote the number of times a simple lie algebra with root system $R$ appears in $\rs{}$ by $\rs{}(R)$.

For example $\rs{}(A_{n})$ is the number of summands which are $\mathfrak{sl}_{n+1}$.

Our goal is to find all those numbers from $\lcs{}(\mathfrak{g})$, notice that as long as we use invariants like $\rm{dim}(\lcs{}^{i})$, which are additive, then if we find one coefficient, then we know its impact on all our invariants, and can practically assume the coefficient is zero.

This suggests trying to find all coefficients by induction, since the number of non-zero $\lcs{}^{i}$ or $\mathfrak{u}^{i}$ depends on the degree of the highest root, it is natural to induct on that.

With that in mind, we list the degree of the highest root for all simple algebras.

\begin{center}
\begin{tabular}{||c c||} 
 \hline
 Lie Algebra & maximal degree  \\ [0.5ex] 
 \hline\hline
 $A_n$ & $n$ \\ 
 \hline
 $B_n$ & $2 n - 1$ \\ 
 \hline
 $C_n$ & $2 n - 1$ \\ 
 \hline
 $D_n$ & $2 n - 3$ \\ 
 \hline
 $E_6$ & $11$ \\ 
 \hline
 $E_7$ & $17$ \\ 
 \hline
 $E_8$ & $29$ \\ 
 \hline
 $F_4$ & $11$ \\ 
 \hline
 $G_2$ & $5$ \\  [1ex] 
 \hline
\end{tabular}
\end{center}

For the proof we will denote $n^{i}$ the dimension of $\lcs{}^{i}$, $u^{i}$ will be the dimension of $\mathfrak{u}^{i}$, and $n^{i,j}$ will be the codimension of the right kernel of $\lcs{}^{i}\times \lcs{}^{j} \rightarrow \lcs{}^{i+j}$, which is $n^j$ minus the dimension of the right kernel.

A few observations will be helpful:

\begin{remark}
$n^{i},u^{i},n^{i,j}$ are zero for simple algebras where the maximal degree is less then $i+j$.
\end{remark}
\begin{proof}

For $n^{i},u^{i}$ this is obvious by our description. 

F.,,$n^{i,j}$ this is because the product takes values in $n^{i+j}$, which is zero. So the kernel is everything, and it's codimension is zero.
\end{proof}
\begin{remark}
Since there is always one highest vector, then for a simple algebra with the degree of the highest root $d$, $n^{d}$ is $1$.

In that case $n^{d-1}$ is the number of roots of degree $d-1$ (which is usually $1$). 

On the other hand, $u^{d}$ is the number of simple roots, and $u^{d-1}$ is the number of roots which are sums of two simple roots.

Two simple roots sum to a root iff they are not orthogonal, i.e. if they are connected in the Dynkin diagram, so we always have $u^{d-1} = u^{d} - 1$.
\end{remark}

\begin{proof}[Proof of Main Theorem]
We will show that using our invariant one can find $\rs{}(R)$ for any simple algebra.

Our proof will be by induction, assuming we know these coefficients for all algebras with the degree of the highest root $\geq 2n$ we will show we can find those with degree of the highest root $2n-1,2n-2$. (A base case is obvious, if $\lcs{k} = 0$ then there are no algebras with roots of degree $\geq k$).

If $2n-1 \neq 1,3,5,11,17,29$ then the only relevant coefficients are $\rs{}(A_{2n-1}),\rs{}(A_{2n-2}),\rs{}(B_{n}),\rs{}(C_{n}),\rs{}(D_{n+1})$. 

The following table depicts $5$ of our invariants, and their value on these simple algebras, and are easily seen to be zero on algebras with lower maximal degree:
\begin{center}
\begin{tabular}{||c c c c c c||} 
 \hline
  & $A_{2n-1}$ & $A_{2n-2}$ & $B_{n}$ & $C_{n}$ & $D_{n+1}$ \\ [0.5ex] 
 \hline\hline
 $n^{2n-1}$ & $1$ & $0$ & $1$ & $1$ & $1$  \\ 
 \hline
 $n^{2n-2}$ & $2$ & $1$ & $1$ & $1$ & $1$ \\ 
 \hline
 $u^{2n-1}$ & $2n-1$ & $0$ & $n$ & $n$ & $n+1$ \\ 
 \hline
 $u^{2n-2}$ & $2n-2$ & $2n-2$ & $n-1$ & $n-1$ & $n$ \\ 
 \hline
 $n^{2,2n-3}$ & $2$ & $0$ & $2$ & $1$ & $2$ \\ 
 \hline
\end{tabular}
\end{center}

One can check by determinant that this matrix is non-degenerate, this can also easily be done by hand.

Ignoring the algebras with higher maximal degree (which we assumed we know their coefficients) we have:
$$u^{2n-1} - u^{2n-2} -  n^{2n-1} = (2n-1) \rs{}(A_{2n-2}) $$
So we know $\rs{}(A_{2n-2})$.

$$n^{2n-1} - n^{2n-2} - \rs{}(A_{2n-2}) = \rs{}(A_{2n-1})$$ so we know this as well.

$u^{2n-1} - n\cdot n^{2n-1}$ is $\rs{}(D_{n+1})$ up to things we know, so now we know everything but $\rs{}(B_{n}),\rs{}(C_{n})$.

But we know their sum from any of the four first equations, and with the last equation, we can find each of them.

Now we have to consider the case where there are also exceptional, we start with $2n-1 = 29$, in this case, we need to also find $\rs{}(E_{8})$, and we will add the invariant $n^{2,28}$. 

\begin{center}
\begin{tabular}{||c c c c c c c||} 
 \hline
  & $A_{29}$ & $A_{28}$ & $B_{15}$ & $C_{15}$ & $D_{16}$ & $E_{8}$ \\ [0.5ex] 
 \hline\hline
 $n^{29}$ & $1$ & $0$ & $1$ & $1$ & $1$& $1$  \\ 
 \hline
 $n^{28}$ & $2$ & $1$ & $1$ & $1$ & $1$ & $1$ \\ 
 \hline
 $u^{29}$ & $29$ & $0$ & $15$ & $15$ & $16$ & $8$ \\ 
 \hline
 $u^{28}$ & $28$ & $28$ & $14$ & $14$ & $15$ & $7$ \\ 
 \hline
 $n^{2,27}$ & $2$ & $0$ & $2$ & $1$ & $2$ & $1$ \\ 
 \hline
 $n^{2,26}$ & $4$ & $2$ & $1$ & $2$ & $1$ & $1$ \\ 
 \hline
\end{tabular}
\end{center}

Again $u^{29} - u^{28} - n^{29}$ gives us $\rs{}(A_{28})$, and then $n^{28} - n^{29}$ gives us $\rs{}(A_{29})$, so we can ignore them. 

This time $3n^{29} - n^{2,27} - n^{2,26}$ gives us $\rs{}(E_{8})$ and from there we can continue as if this was the previous case.

$2n-1 = 17$:

\begin{center}
\begin{tabular}{||c c c c c c c||} 
 \hline
  & $A_{17}$ & $A_{16}$ & $B_{9}$ & $C_{9}$ & $D_{10}$ & $E_{7}$ \\ [0.5ex] 
 \hline\hline
 $n^{17}$ & $1$ & $0$ & $1$ & $1$ & $1$& $1$  \\ 
 \hline
 $n^{16}$ & $2$ & $1$ & $1$ & $1$ & $1$ & $1$ \\ 
 \hline
 $u^{17}$ & $17$ & $0$ & $9$ & $9$ & $10$ & $7$ \\ 
 \hline
 $u^{16}$ & $16$ & $16$ & $8$ & $8$ & $9$ & $6$ \\ 
 \hline
 $n^{2,15}$ & $2$ & $0$ & $2$ & $1$ & $2$ & $1$ \\ 
 \hline
 $n^{2,14}$ & $4$ & $2$ & $1$ & $2$ & $1$ & $1$ \\ 
 \hline
\end{tabular}
\end{center}
This case is essentially identical to the last case.

$2n-1 = 11$:
This case has two exceptionals, so we need another invariant, we will add $n^{3,8}$.

\begin{center}
\begin{tabular}{||c c c c c c c c||} 
 \hline
  & $A_{11}$ & $A_{10}$ & $B_{6}$ & $C_{6}$ & $D_{7}$ & $E_{6}$ & $F_{4}$  \\ [0.5ex] 
 \hline\hline
 $n^{11}$ & $1$ & $0$ & $1$ & $1$ & $1$ & $1$ & $1$  \\ 
 \hline
 $n^{10}$ & $2$ & $1$ & $1$ & $1$ & $1$ & $1$  & $1$ \\ 
 \hline
 $u^{11}$ & $11$ & $0$ & $6$ & $6$ & $7$ & $6$ & $4$ \\ 
 \hline
 $u^{10}$ & $10$ & $10$ & $5$ & $5$ & $6$ & $5$ & $3$ \\ 
 \hline
 $n^{2,9}$ & $2$ & $0$ & $2$ & $1$ & $2$ & $1$ & $1$ \\ 
 \hline
 $n^{2,8}$ & $4$ & $2$ & $1$ & $2$ & $1$ & $2$ & $1$ \\ 
 \hline
 $n^{3,8}$ & $2$ & $0$ & $2$ & $1$ & $2$ & $2$ & $1$\\ 
 \hline
\end{tabular}
\end{center}

In the same way as before we get $\rs{}(A_{11}),\rs{}(A_{10})$.

Then $n^{3,8}-n^{2,9}$ gives us $\rs{}(E_{6})$.

$3n^{11} - n^{2,9}- n^{2,8}$ gives us $\rs{}(F_{4})$, and then we are left with only non-exceptionals so we can finish like in the original case.

We are left with algebras with a degree of the highest root $\leq 5$, sadly this case turns out to be the messiest.

It turns out that the invariants we used thus far are not enough, even if we take all of the $n^{i},u^{i},n^{i,j}$, and consider them as equations on the $10$ missing coefficients the solutions are not unique.

To fix this we consider the product $\lcs{}^{i} \times \mathfrak{u}^{j} \rightarrow \mathfrak{u}^{j-i}$, which we have not used thus far, like with the other product we take its codimension of the right kernel, which we denote $nu^{i,j}$.
\begin{center}
\begin{tabular}{||c c c c c c c c c c c||} 
 \hline
  & $A_{5}$ & $A_{4}$ & $B_{3}$ & $C_{3}$ & $D_{4}$ & $G_{2}$  & $A_{3}$ & $B_{2}$ & $A_{2}$ & $A_{1}$\\ [0.5ex] 
 \hline\hline
 $n^{5}$ & $1$ & $0$ & $1$ & $1$ & $1$ & $1$ & $0$ & $0$ & $0$ & $0$  \\ 
 \hline
 $n^{4}$ & $2$ & $1$ & $1$ & $1$ & $1$ & $1$ & $0$ & $0$ & $0$ & $0$ \\ 
 \hline
 $n^{3}$ & $3$ & $2$ & $2$ & $2$ & $3$ & $1$ & $1$ & $1$ & $0$ & $0$ \\ 
 \hline
 $n^{2}$ & $4$ & $3$ & $2$ & $2$ & $3$ & $1$ & $2$ & $1$ & $1$ & $0$ \\ 
 \hline
 $n^{1}$ & $5$ & $4$ & $3$ & $3$ & $4$ & $2$ & $3$ & $2$ & $2$ & $1$ \\ 
 \hline
 $u^{5}$ & $5$ & $0$ & $3$ & $3$ & $4$ & $2$ & $0$ & $0$ & $0$ & $0$ \\ 
 \hline
 $u^{4}$ & $4$ & $4$ & $2$ & $2$ & $3$ & $1$ & $0$ & $0$ & $0$ & $0$ \\ 
 \hline
 $u^{3}$ & $3$ & $3$ & $2$ & $2$ & $3$ & $1$ & $3$ & $2$ & $0$ & $0$ \\ 
 \hline
 $n^{2,3}$ & $2$ & $0$ & $2$ & $1$ & $3$ & $1$ & $0$ & $0$ & $0$ & $0$ \\ 
 \hline
 $nu^{2,3}$ & $2$ & $2$ & $2$ & $3$ & $1$ & $2$ & $1$ & $0$ & $0$ & $0$ \\ 
 \hline
\end{tabular}
\end{center}
At this point showing by hand that the matrix is non-degenerate is unproductive.

By entering the matrix into any matrix calculator one can see that it is non-degenerate (its determinant is $-4$), and so we can recover the missing coefficients, and the proof is complete.
\end{proof}

\medskip
\printbibliography
\end{document}